\documentclass[11pt,draft,book]{amsart}
\usepackage[english]{babel}
\usepackage{amssymb}
\usepackage{graphicx}
\usepackage{geometry}
 \usepackage{epstopdf}
\usepackage{hyperref}  
\hypersetup{colorlinks=true, linkcolor=blue, anchorcolor=blue, 
citecolor=red, filecolor=blue, menucolor=blue,
urlcolor=blue}

\numberwithin{equation}{section}
\newtheorem{theorem}{Theorem}[section]

\newtheorem{proposition}[theorem]{Proposition}

\theoremstyle{remark}
\newtheorem{remark}{Remark}[section]

\theoremstyle{definition}

\newcommand{\bra}[1]{\langle #1 \rangle}

\def\XXint#1#2#3{{\setbox0=\hbox{$#1{#2#3}{\int}$ }
\vcenter{\hbox{$#2#3$ }}\kern-.58\wd0}}

\begin{document}
\title
[The cubic nonlinear Dirac equation]
{The cubic nonlinear Dirac equation}
\begin{abstract}
We present some results obtained in collaboration with prof. Piero D'Ancona concerning global existence for the 3D cubic non linear massless Dirac equation with a potential for small initial data in $H^1$ with slight additional assumptions. The new crucial tool is given by the proof of some refined endpoint Strichartz estimates.
\end{abstract}
\author{Federico Cacciafesta}










\maketitle


\section{Introduction}

We consider the 3D massless Dirac equation perturbed with a small potential $V$
\begin{equation}\label{pb}
\begin{cases}
iu_t+\mathcal{D}u+Vu=P_3(u) \qquad u:\mathbb{R}_t\times\mathbb{R}_x^3\rightarrow \mathbb{C}^4\\
u(0)=f(x)\in H^s,
\end{cases}
\end{equation}
where the \emph{Dirac}  operator is defined as 
$$
\mathcal{D}=1/i\displaystyle\sum_{k=1}^3\alpha_k\partial_k=-i(\alpha\cdot\nabla)
$$ 
with 
\begin{equation*}
\alpha_1=\left(\begin{array}{cccc}0 & 0 &0 & 1 \\0 & 0 &1 & 0\\
0 & 1 &0 & 0\\1 & 0 &0 & 0\end{array}\right),\quad
\alpha_2=\left(\begin{array}{cccc}0 & 0 &0 & -i \\0 & 0 &i & 0\\
0 & -i &0 & 0\\i & 0 &0 & 0\end{array}\right),\quad
\alpha_3=\left(\begin{array}{cccc}0 & 0 &1 & 0 \\0 & 0 &0 & -1\\
1 & 0 &0 & 0\\0 & -1 &0 & 0\end{array}\right),
\end{equation*}
the nonlinear term $P_3(u)\cong |u|^3$ and $V$ is a $4\times 4$ hermitian matrix. Notice that 
\begin{equation}\label{comm11}
\alpha_j\alpha_k+\alpha_k\alpha_j=2\delta_{jk}\bold{1},\quad j,k=1,2,3;
\end{equation}
The unperturbed nonlinear Dirac equation 
is important in relativistic quantum mechanics, and
was studied in a number of works 
(see e.g. \cite{Reed76-a},
\cite{DiasFigueira87-a},
\cite{Najman92-a},
\cite{Moreau89-a},
\cite{EscobedoVega97-a},
\cite{MachiharaNakamuraOzawa04-a},
\cite{MachiharaNakamuraNakanishi05-a}). 
In particular, it is well known that the cubic nonlinearity is \emph{critical}\footnote{The general framework when dealing with low regularity well posedness for nonlinear dispersive equations is to be thought as follows: the scaling of the operator, the functional setting and the structure of the nonlinearity yield a threshold regularity such that if the initial data has more regularity than this (i.e. is \emph{subcritical}) it is possible to prove global wellposedness with standard contraction arguments. If the initial data is instead \emph{supercritical} then we have various degrees of illposedness, while the \emph{critical} case is typically harder and many different things can happen.}
for solvability in the energy space $\dot{H}^{1}$;
global existence in $\dot{H}^{1}$
is still an open problem even for small initial data, while
the case of subcritical spaces $\dot{H}^{s}$, $s>1$
was settled in the positive in
\cite{EscobedoVega97-a},
\cite{MachiharaNakamuraOzawa04-a}. 
The standard tool when proving global wellposedness for general dispersive equations in a low regularity setting, at least in the subcritical case, relies on some space-time estimates (\emph{Strichartz estimates}) that allow the application of the classical fixed point method on some suitable functional space; the critical case may be corresponding to the endpoint estimate and thus may require a different approach. 

To begin with, let us give a look to the free homogeneous case.
Notice that property \eqref{comm11} yields $\mathcal{D}^2=-I_4\Delta$, so that
\begin{equation}\label{rap}
(i\partial_t-\mathcal{D})(i\partial_t+\mathcal{D})u=-\square u.
\end{equation}
This means that every solution to the free Dirac equation is a solution to a vectorial wave equation with suitable initial conditions. Despite the rich structure of $\mathcal{D}$ it is thus easy, relying on the well known theory for the free wave equation, to write the Dirac linear propagator as
\begin{equation}\label{dirsol}
e^{it\mathcal{D}}f=\cos(t|D|)f+i\frac{\sin(t|D|)}{|D|}\mathcal{D}f.
\end{equation}
 and the Strichartz estimates
\begin{equation}\label{str}
\| e^{it\mathcal{D}}f\|_{L^p_t\dot{H}_q^{\frac{1}{q}-\frac{1}{p}-\frac{1}{2}}}\lesssim\|f\|_{L^2}
\end{equation}
that hold for every couple $(p,q)$ such that
\begin{equation*}
\frac{2}{p}+\frac{2}{q}=1,\quad 2\leq p\leq\infty,\qquad\infty> q\geq2.
\end{equation*}
We recall that estimates \eqref{str} capture two different informations:
\\\\ \emph{Locally in time} they describe a type of smoothing effect, reflected in a gain of integrability with respect to some suitable norm.
 \\\\ \emph{Globally in time} they describe a decay effect, meaning that some spacial norm of a solution must decay to zero as $t\rightarrow\infty$, at least in some $L^p$ time-averaged sense.
\\\\
\indent The choice of the couple $(p,q)=(2,\infty)$ yields the so called \emph{endpoint estimate}
\begin{equation}\label{strend}
\| e^{it\mathcal{D}}f\|_{L^2_tL^\infty_x}\lesssim\|f\|_{\dot{H}^1}
\end{equation}
that is known to fail, as the corresponding one for the 3D wave flow (see \cite{klai-mach}). If \eqref{strend} were true one could prove global wellposedness for problem \eqref{pb} (at least for $V=0$) with standard methods. Considering the solution map
\begin{equation*}
  v\mapsto \Phi(v)=e^{it \mathcal{D}}f
  +i\int_{0}^{t}e^{i(t-t')\mathcal{D}}P_{3}(v(t'))dt'
\end{equation*}
one would have indeed
\begin{equation*}
  \left\|
    \int_{0}^{t}e^{i(t-t')\mathcal{D}}v(t')^{3}dt
  \right\|_{L^{2}L^{\infty}}
  \lesssim
  \int_{-\infty}^{+\infty}
  \|e^{it \mathcal{D}}e^{-it' \mathcal{D}}P_{3}(v(t'))\|
  _{L^{2}L^{\infty}}dt'
  \lesssim
  \|v^{3}\|_{L^{1}H^{1}}
\end{equation*}
that in conjuction with the conservation of $H^{1}$ norm would imply
\begin{equation*}
  \|\Phi(v)\|_{L^{\infty}_{t}H^{1}_{x}}+
     \|\Phi(v)\|_{L^2_{t}L^{\infty}_{x}}\lesssim
  \|f\|_{H^{1}}+\|v\|_{L^{\infty}H^{1}}\|v\|_{L^{2}L^{\infty}}^{2}.
\end{equation*}
In other words, a contraction argument in the norm
$\|\cdot\|_{L^{2}L^{\infty}}+\|\cdot\|_{L^{\infty}H^{1}}$ would be
enough to prove global existence of small $H^{1}$ solutions to
\eqref{pb}.

The failure of estimate \eqref{strend} prevents us from the application of this machinery. Our aim here will thus be to provide some "weak endpoint" estimates, i.e. some suitable refinements of estimate \eqref{strend} in view of obtaining global wellposedness for problem \eqref{pb} in the critical case with some slight additional hypothesis on the initial data. 

We shall here present two approaches to the problem, that in both cases will prove, at least, global existence for problem \eqref{pb} for \emph{radial} $H^1$ initial data.

\section{An "algebraic" approach}

In \cite{klai-mach} the authors noticed that the endpoint estimate for the 3D wave flow can be recovered if the initial data are assumed to be radial. As it can be easily seen indeed, for a radial $f$ we have, for all $x$,
\begin{equation*}
\frac{\sin(t|D|)}{|D|}f=\frac c{|x|}\int^{|x|+t}_{||x|-t|}sf(s)ds\lesssim M(g)(t)
\end{equation*}
where $M(g)$ is the Hardy maximal function of the function $g(s)=sf(s)$, so that
\begin{equation*}
\left\|\frac{\sin(t|D|)}{|D|}\right\|_{L^\infty_x}\lesssim M(g)(t)
\end{equation*}
and thus by a standard maximal estimate
\begin{equation*}
\left\|\frac{\sin(t|D|)}{|D|}\right\|_{L^2_tL^\infty_x}\lesssim \|g\|_{L^2(\mathbb{R}}=\|sf(s)\|_{L^2_s(\mathbb{R}}\cong\|f\|_{L^2(\mathbb{R}^3)}
\end{equation*}
(analogous arguments hold in dimension $n\geq 3$, see \cite{FangWang06-a}). 
This remark suggests the chance of slightly improving the range of admissible exponents when requiring some additional structure on the initial data. Even if the Dirac operator does not preserve the radiality of functions, we are indeed able to prove the following

\begin{proposition}\label{prop1}
Let $n=3$ ad let $f$ belong to the space 
\begin{equation}\label{H1}
\mathcal{\dot{H}}^1=\{f_1+\mathcal{D}f_2,\;f_1\in\dot{H}^1(\mathbb{R}^{3}),
\;f_2\in\dot{H}^2(\mathbb{R}^{3}),
\;f_1,f_2\;
\rm{radial}\}.
\end{equation}
Then the following endpoint Strichartz estimate holds:
\begin{equation}\label{homdir}
\| e^{it\mathcal{D}}f\|_{L^2_tL^\infty_x}\lesssim
\| f\|_{\dot H^1}.
\end{equation}
\end{proposition}

Our proof of this result, already contained in \cite{MachiharaNakamuraNakanishi05-a} and mainly relying on Fourier transform in radial coordinates, takes the remarkable advantage of being adaptable also to handle the non homogenous term. With minor modifications we are in fact able to prove the following mixed Strichartz-smoothing estimate

\begin{proposition}\label{prop2}
Let $n=3$ and assume $F(t,x)$ has the structure
\begin{equation}\label{Fform}
F(t,x)=F_1(|x|)\mathbb{I}_4+i(\alpha\cdot\hat{x})F_2(|x|).
\end{equation}
Then the following estimate holds
\begin{equation}\label{stdir}
\displaystyle\left\|\int_0^t e^{i(t-s)\mathcal{D}}F(s)\:ds\right\|_{L^2_tL^\infty_x}\lesssim
\| \langle x\rangle^{\frac{1}{2}+}|D| F\|_{L^2_tL^2_x}.
\end{equation}
\end{proposition}

This result is crucial in view of obtaining endpoint Strichartz estimates for the potential-perturbed Dirac flow, and to keep the loss of derivatives less or equal than 1. Propositions \ref{prop1} and \ref{prop2} yield in fact the following Theorem, a proof of which will be also sketched.

\begin{theorem}[Endpoint estimate for the perturbed Dirac flow]\label{teo1}
Let $V(x)$ be a $4\times 4$ matrix of the form 
\begin{equation}\label{pppot}
V(x)=V_1(|x|)\mathbb{I}_4+i\beta(\alpha\cdot\hat{x})V_2(|x|),\qquad
V_1,V_2:\mathbb{R}^{+}\to \mathbb{R}.
\end{equation}
Assume that for some $\sigma>1$ and some sufficiently small
$\delta>0$
\begin{equation}\label{Vhp}
  |V(x)|\le \frac{\delta}
  {\langle x\rangle^{1/2+}w_\sigma^{1/2}},\quad 
   |\nabla V(x)|\le \frac{\delta}
  {\langle x\rangle^{1/2+}w_\sigma^{1/2}},
\end{equation}
with $w_\sigma(x)^{1/2}=|x|(1+|\log|x||)^\sigma$.\\
Then the following endpoint Strichartz estimate
\begin{equation}\label{endV}
\| e^{it(\mathcal{D}+V)}f\|_{L^2_t L^\infty_x}\lesssim
\| f\|_{ H^1}
\end{equation}
holds for all initial data $f\in\mathcal{\dot H}^1\cap H^1$.
\end{theorem}

\begin{proof}
By Duhamel's formula we write the integral form
\begin{equation*}
  u=e^{it(\mathcal{D}+V)}f=
  e^{it \mathcal{D}}f+
  i\int_{0}^{t}e^{i(t-s)\mathcal{D}}(Vu)ds
\end{equation*}
so that we have
\begin{equation*}
\|u \|_{L^2_tL^\infty_x}=
\|e^{it(\mathcal{D}+V)}f \|_{L^2_tL^\infty_x}
\end{equation*}
$$
\leq\left\|e^{it\mathcal{D}}f\right\|_{L^2_tL^\infty_x}+
\left\|\int_0^t 
e^{i(t-s)\mathcal{D}}
(V(s)e^{is(\mathcal{D}+V)}f)ds\right\|_{L^2_tL^\infty_x}$$
Applying Propositions \ref{prop1} and \ref{prop2} we can continue estimating with
$$
\lesssim\|f\|_{\dot H^1}+\|\langle x\rangle^{\frac{1}{2}+}|D|(Vu)\|_{L^2_tL^2_x}.
$$
From the $L^2$ boundedness of the Riesz operator $|D|^{-1}\nabla$ we have
$$
\lesssim \|f\|_{\dot H^1}+\|\langle x\rangle^{\frac{1}{2}+}(\nabla V)u\|_{L^2_tL^2_x}
+\|\langle x\rangle^{\frac{1}{2}+}V(\nabla u)\|_{L^2_tL^2_x}.
$$
To conclude we need the following smoothing estimates (the first of which is proved in D'Ancona-Fanelli '08, the second directly follows), which hold with more general assumptions on the potential $V$:
\begin{itemize}
\item
$\|w_\sigma^{-1/2}e^{it(\mathcal{D}+V)}f\|_{L^2_tL^2_x}\lesssim\|f\|_{L^2}$,
\item
$\|w_\sigma^{-1/2}\nabla e^{it(\mathcal{D}+V)}f\|_{L^2_tL^2_x}\lesssim\|f\|_{\dot H^1}.
$
\end{itemize}
Thus multiplying and dividing by 
$$
w_\sigma(x)^{1/2}=|x|(1+|\log|x||)^\sigma
$$ 
in the smoothing terms yields
\begin{itemize}
\item
$
\|\langle x\rangle^{\frac{1}{2}+}(\nabla V)u\|_{L^2_tL^2_x}\lesssim
\|\langle x\rangle^{1/2+}w_\sigma^{1/2} \nabla V\|_{L^\infty}\cdot\| w_\sigma^{-1/2} u\|_{L^2_tL^2_x}
$
$$
\lesssim  \|f\|_{L^2},
$$
\item
$
\|\langle x\rangle^{\frac{1}{2}+}V(\nabla u)\|_{L^2_tL^2_x}\lesssim
\|\langle x\rangle^{1/2+}w_\sigma^{1/2} V\|_{L^\infty}\cdot\| w_\sigma^{-1/2} \nabla u\|_{L^2_tL^2_x}
$
$$
\lesssim  \|f\|_{\dot H^1}
$$
\end{itemize}
and this concludes the proof.
\end{proof}

\begin{remark}
Notice that hypothesis \eqref{Vhp}, that is needed in view of applying proposition \ref{prop2}, is fairly natural: the structure required describes in fact electric potentials and particles with anomalous magnetic momentum.
\end{remark}
It would be now tempting to apply this result to prove global existence for problem \eqref{pb} for small initial data in $\mathcal{\dot H}^1\cap H^1$  with standard fixed-point techniques; 
since such a method is iterative, what we would need is that class to be invariant under the action of the cubic non linearity $P_3(u)$, but this fact is unfortunately not true even with the simplest choice $P_3(u)=\langle \beta u,u\rangle u$.

Our next step will thus be to show the existence of some proper subspace of the set $\mathcal{\dot H}^1\cap H^1$ on which the following conditions are satisfied:
\begin{itemize}
\item 
The operator $\mathcal{D}+V$ is well-defined and selfadjoint,
\item
The action of $P_3(u)$ is invariant.  
\end{itemize}

The classical theory of \emph{partial wave subspaces} (see \cite{thaller}) allows in fact, analyzing the structure of the Dirac operator in radial coordinates and using spherical harmonics, to decompose the space $L^2(\mathbb{R}^3)^4$ in a direct sum of Hilbert spaces that are left invariant under the action of the potential perturbed Dirac operator. Moreover, an explicit insight of these spaces shows that "some of them" are left invariant by the action of the standard cubic non linearity as well. We collect all this results in the following proposition, referring to \cite{thaller} and \cite{Cacciafesta1} for notation, details of the statement and a proof.

\begin{proposition}\label{dir+v}
There exists a family of 2-dimensional Hilbert spaces $\mathcal{H}_{m_j,k_j}$ (the "partial wave subspaces") such that
\begin{equation*}
L^2(\mathbb{R}^3)^4\cong \bigoplus L^2((0,\infty),dr)\otimes \mathcal{H}_{m_j,k_j}
\end{equation*}
and the operator $\mathcal{D}+V$, with $V$ of the form \eqref{pppot} acts and is self-adjoint on such spaces.

Moreover for $j=1/2$ the spaces $\mathcal{H}_{m_{1/2},k_{1/2}}$ are left invariant also by the action of the cubic nonlinearity $P_3(u)=\langle \beta u,u\rangle u$.
\end{proposition}

\begin{remark}
Each of these partial wave subspaces has a basis $\{\Phi_{m_j,k_j}^+,\Phi_{m_j,k_j}^-\}$, that can be explicitly written using Legendre polynomials, satisfying the property that $$\mathcal{D}\Phi_{m_j,k_j}^{\pm}=\mp\Phi_{m_j,k_j}^\mp.$$
The action of the operator $\mathcal{D}+V$ can thus be easily represented in coordinates with respect to such basis, and in this contest is commonly referred to as the \emph{radial Dirac operator}.
\end{remark}

\begin{remark}
The second part of the statement is justified by the fact that for $j=1/2$ it is easy to evaluate the action of $P_3$ on the functions $\Phi_{m_j,k_j}^+$, $\Phi_{m_j,k_j}^-$, that in this particular case have a fairly simple structure. To give a clearer picture of the situation, consider a.e. the triple\footnote{Notice that for a fixed $j=\frac12, \frac32,...$ there will be $4j+2$ possible choices of $m_j$, $k_j$. More precisely the range is $m_j=-j,-j+1,\dots,+j$ and $k_j=-(j+1/2),+(j+1/2)$.} $(j,m_j,k_j)=(1/2,1/2,1)$; then we have
\begin{equation*}
\Phi^+_{1/2,1}=
\left(\begin{array}
{cc}
\displaystyle\frac{i}{2\sqrt\pi}\cos\theta\\
\displaystyle
\frac{i}{2\sqrt\pi}e^{i\phi}\sin\theta\\
0\\0
\end{array}\right)\qquad
\Phi^-_{1/2,1}=
\left(\begin{array}
{cc}
0\\0\\
\displaystyle\frac{1}{2\sqrt\pi}\\0
\end{array}\right).
\end{equation*}
As it is easily seen, this yields the fact that for a generic function $u\in L^2((0,\infty),dr)\otimes  \mathcal{H}_{1/2,1}$ that will thus be written as
\begin{equation*}
u(r,\theta,\phi)=u^+(r)\Phi^+_{1/2,1}(\theta,\phi)+u^-(r)\Phi^-_{1/2,1}(\theta,\phi)
\end{equation*}
for some radial functions $u^+$, $u^-$, the nonlinear term reads as
\begin{eqnarray*}
\langle\beta u,u\rangle&=&\displaystyle-\frac{1}{4\pi}\cos^2\theta\: u^+(r)^2
-\frac{1}{4\pi}\sin^2\theta\: u^+(r)^2-\frac{1}{4\pi}u^-(r)^2=\\
&=&\displaystyle -\frac{1}{4\pi}\left(u^+(r)^2+u^-(r)^2\right),
\end{eqnarray*}
and so has no angular component. This proves that the action of the standard cubic nonlinearity is invariant on the space 
$\mathcal{H}_{1/2,1}$.
\end{remark}

This last remark provides the final tool we needed to state our global existence result, that now can be proved with completely standard techniques.

\begin{theorem}\label{teo2}
Let $P_3(u)=\langle\beta u,u\rangle u$ and the potential $V$ satisfying \eqref{pppot}- \eqref{Vhp}. Then for every initial data $f\in\dot H^1((0,\infty),dr)\otimes  \mathcal{H}_{m_{1/2},k_{1/2}}$, with sufficiently small $\dot H^1$ norm, 
there exists a unique global solution $u(t,x)$ to problem 
(\ref{pb}) in the class 
$C_t(\mathbb{R},\dot H^1)\cap L^2_t(\mathbb{R},L^\infty).$
\end{theorem}

\section{Angular regularity estimates}

In recent years many papers (see e.g. \cite{sterb}, \cite{FangWang08-a}, \cite{JiangWangYu10-a}, \cite{MachiharaNakamuraNakanishi05-a} and references therein) have been devoted to the improvement of Strichartz estimates using angular regularity. In particular in \cite{MachiharaNakamuraNakanishi05-a} it is proved the following estimate for the wave propagator
\begin{equation}\label{eq:mnnop}
  n=3,\qquad
  \|e^{it |D|}f\|
      _{L^{2}_tL^{\infty}_{r}L^{p}_{\omega}}\lesssim
  \sqrt{p}\cdot
  \||D|f\|_{L^{2}},\qquad \forall p<\infty
\end{equation}
where we are using the natural notation
\begin{equation*}
  \|f\|_{L^{a}_{r}L^{b}_{\omega}}=
  \left(
    \int_{0}^{\infty}\|f(r\ \cdot\ )\|_{L^{b}(\mathbb{S}^{n-1})}
    ^{a}r^{n-1}dr
  \right)^{\frac1a}
\end{equation*}
and
\begin{equation*}
  \|f\|_{L^{\infty}_{r}L^{b}_{\omega}}=
  \sup_{r\ge0}
    \|f(r\ \cdot\ )\|_{L^{b}(\mathbb{S}^{n-1})}.
\end{equation*}
Notice that the norm at the left hand side distinguishes between the
integrability in the radial and tangential
directions. Using estimate \eqref{eq:mnnop}, Machihara et al. were able
to prove global well posedness for problem \eqref{pb} with $V=0$ for small $\dot{H}^1$-norm data with slight additional angular regularity,
and in particular for all \emph{radial} $\dot H^{1}$ data. This
is especially interesting since, as pointed out in the introduction, radial data
do not correspond to radial solution for the Dirac equation (due to the fact that the operator
$\mathcal{D}$ does not commute with rotations of $\mathbb{R}^{3}$).

Estimate \eqref{eq:mnnop} gives a bound for the standard $L^{2}L^{\infty}$ norm
via Sobolev embedding on the unit sphere $\mathbb{S}^{2}$
\begin{equation}\label{eq:refined}
  \|e^{it |D|}f\|_{L^{2}L^{\infty}}\lesssim
  \|\Lambda^{\epsilon}_{\omega} e^{it |D|}f\|
      _{L^{2}L^{\infty}_{r}L^{p}_{\omega}}
  \lesssim
  \||D|\Lambda^{\epsilon}_{\omega}f\|_{L^{2}}, 
  \qquad
  p>\frac{2}{\epsilon}
\end{equation}
where the angular derivative operator $\Lambda_{\omega}^{s}$
is defined in terms of the Laplace-Beltrami operator on 
$\mathbb{S}^{n-1}$ as
$\Lambda^{s}_{\omega}=(1-\Delta_{\mathbb{S}^{n-1}})^{s/2}$.

Our main goal here is to extend this group of results to the Dirac
equation perturbed with a small potential
$V(x)$. We consider first the linear equation
\begin{equation}\label{eq:linearV}
  iu_{t}=\mathcal{D}u+V(x)u+F(t,x).
\end{equation}
The perturbative term $Vu$ can not be handled using the
inhomogeneous version of \eqref{eq:mnnop} because of the
loss of derivatives. Instead, we shall need to prove new mixed 
Strichartz-smoothing estimates.

We collect in the following Theorem the Strichartz estimates we are able to prove both for the wave and Dirac equations 
(see \cite{Cacciafesta3}).

\begin{theorem}\label{the:strichom}
Let $n\geq 3$ and 
\begin{equation}\label{sig}
\sigma_n=\begin{cases}
0 \quad\text{if $n=3$}\\
  1-\frac n2 \quad\text{if $n\geq4$}.
    \end{cases}
    \end{equation}
    Then for every $s\geq0$ the following estimates for the free wave equation hold
  \begin{equation}\label{eq:strichartz1}
    \| \Lambda^s_\omega e^{it|D|}f\|_{L^2_tL^\infty_{r}L^2_\omega}
    \lesssim
    \|\Lambda^{s+\sigma_n}_{\omega} f\|_{\dot H^\frac{n-1}{2}};
  \end{equation}
    \begin{equation}\label{non1}
    \left\|\Lambda^s_\omega
      \int_0^t e^{i(t-s)|D|}F(s,x)ds
    \right\|_{L^2_tL^\infty_rL^2_\omega}
    \lesssim
    \|
      \bra{x}^{\frac{1}{2}+}|D|^{\frac{n-1}{2}}\Lambda_{\omega}^{s+\sigma_n}F
    \|_{L^2_tL^2_{x}}.
  \end{equation}
 The corresponding estimates for the 3D free Dirac equation are
\begin{equation}\label{eq:freedirac}
    \| \Lambda_{\omega}^{s} 
         e^{it \mathcal{D}}f\|_{L^2_tL^\infty_{r}L^2_\omega}
    \lesssim
    \|\Lambda_{\omega}^{s} f\|_{\dot H^1},
  \end{equation}
 
  \begin{equation}\label{non2}
    \left\|\Lambda_{\omega}^{s} 
      \int_0^t e^{i(t-t')\mathcal{D}}F(t',x)dt'
    \right\|_{L^2_tL^\infty_rL^2_\omega}
    \lesssim
    \|
      \bra{x}^{\frac{1}{2}+}|D|\Lambda_{\omega}^{s} F
    \|_{L^2_tL^2_{x}}.
  \end{equation}

\end{theorem}

\begin{remark}\label{rem:comparehomog}
  As a byproduct of our proof we have thus obtained
  the following endpoint estimates for the wave flow
  with gain of angular regularity:
  \begin{equation}\label{eq:ourwave}
    n\ge3,\qquad
    \| e^{it|D|}f\|_{L^2_tL^\infty_{r}L^2_\omega}
    \lesssim
    \|\Lambda^{\sigma_n}_{\omega} f\|_{\dot H^\frac{n-1}{2}}
  \end{equation}
  where $\sigma_n$ is as in \eqref{sig}.
  Although this was not the main purpose here,
  it is interesting to compare
  \eqref{eq:ourwave} with known results. In dimension $n=3$,
  estimate \eqref{eq:ourwave} is just a special case of 
  Theorem 1.1-III in \cite{MachiharaNakamuraNakanishi05-a} 
  where \eqref{eq:ourwave} is proved with
  $\sigma_3=-\frac34$; it is not known if this value
  is sharp, however in the same paper it is proved that
  the estimate is false for $\sigma_3<-\frac56$.
  On the other hand, to our knowledge,
  estimate \eqref{eq:ourwave} for $n\ge4$ and \eqref{non1}
  for $n\ge3$ are new. The literature on these kind
  of estimates is extensive and we refer to
  \cite{FangWang08-a}, 
  and the references therein for further information.
\end{remark}

The next step is to prove suitable smoothing estimates
for the Dirac equation with potential. By a perturbative argument we
obtain the following endpoint estimates for the linear flows:

\begin{theorem}\label{the:strichD-i}
  Assume that the hermitian matrix $V(x)$ satisfies,
  for $\delta$ sufficiently small, $C$ arbitrary and $\sigma>1$, with
  $v(x)=|x|^{\frac12}|\log|x||^{\frac12+}+|x|^{1+}$,
  \begin{equation}\label{eq:assnablaV2-i}
    |V(x)|\leq\frac{\delta}
         {v(x)},\qquad
    |\nabla V(x)|\leq\frac{C}
         {v(x)}.    
  \end{equation}
  Then the perturbed Dirac flow
  satisfies the endpoint Strichartz estimate
  \begin{equation}\label{eq:enddiracV-i}
    \|e^{it(\mathcal{D}+V)}f\|_{L^2_t L^\infty_{r}L^2_\omega}
    \lesssim
    \| f\|_{H^1}.
  \end{equation}
  If the potential satisfies the stronger assumptions: 
  for some $s>1$,
  \begin{equation}\label{eq:nablaangV2-i}
    \|\Lambda_{\omega}^{s}
         V(|x|\ \cdot\ )\|_{L^{2}(\mathbb{S}^{2})}
    \le \frac{\delta}{v(x)},
    \qquad
    \|\Lambda_{\omega}^{s}\nabla
         V(|x|\ \cdot\ )\|_{L^{2}(\mathbb{S}^{2})}
    \le \frac{C}{v(x)},
  \end{equation}
  then we have the endpoint estimate with angular regularity
  \begin{equation}\label{eq:enddiracVang-i}
    \|\Lambda^{s}_{\omega}
      e^{it(\mathcal{D}+V)}f\|_{L^2_t L^\infty_{r}L^2_\omega}
    \lesssim
    \|\Lambda^{s}_{\omega} f\|_{H^1}
  \end{equation}
  and the energy estimate with angular regularity
  \begin{equation}\label{eq:energyang-i}
    \|\Lambda^{s}_{\omega}
      e^{it(\mathcal{D}+V)}f\|_{L^\infty_t H^{1}}
    \lesssim
    \|\Lambda^{s}_{\omega} f\|_{H^1}
  \end{equation}
\end{theorem}

\begin{remark}\label{vg}
Notice the difference in the assumptions on $V$ between Theorems \ref{teo1} and \ref{the:strichD-i}: the hypothesis on structure \eqref{pppot} here is substituted by \eqref{eq:nablaangV2-i} which ensures some minimal angular regularity needed.
\end{remark}

We can finally apply Theorem \ref{the:strichD-i} to the
nonlinear equation \eqref{pb} to obtain:

\begin{theorem}\label{the:globalNL-i}
  Consider the perturbed Dirac system \eqref{pb}, where
  the $4\times4$ matrix valued potential $V(x)$
  is hermitian and satisfies assumptions
  \eqref{eq:nablaangV2-i}.
  Let $P_{3}(u,\overline{u})$ be a $\mathbb{C}^{4}$-valued 
  homogeneous cubic polynomial.
  Then for any $s>1$ there exists $\epsilon_{0}$ such that
  for all initial data satisfying
  \begin{equation}\label{eq:data}
    \|\Lambda_{\omega}^{s}f\|_{H^{1}}<\epsilon_{0}
  \end{equation}
  the Cauchy problem \eqref{pb} admits a unique global solution
  $u\in CH^{1}\cap L^{2}L^{\infty}$
  with $\Lambda^{s}_{\omega}u\in L^{\infty}H^{1}$.
\end{theorem}

In particular, problem \eqref{pb} 
has a global unique solution for all radial
data with sufficiently small $H^{1}$ norm.

\begin{remark}\label{rem:wave}
  It is clear that our methods can also be applied to
  nonlinear wave equations perturbed with potentials, and
  allow to prove global well posedness for some types of
  critical nonlinearities.\end{remark}

\begin{remark}\label{rem:generality}
  We did not strive for the sharpest condition on the
  potential $V$, which can be improved
  at the price of additional technicalities which
  we prefer to skip here.  
  Moreover, differently from the previous chapter in which the structure of the non linear term was essential, the result can be extended
  to more general cubic nonlinearities $|P_3(u)|\sim |u|^{3}$.
  \end{remark}
  
  \begin{remark}\label{ss}
  Notice that we need an angular regularity
  $s>1$ on the data, higher than the $s>0$ assumed
  in  the result of \cite{MachiharaNakamuraNakanishi05-a}. 
  It is possible to relax our assumptions to $s>0$;
  the only additional tool we would need to prove is
  a Moser-type product estimate
  \begin{equation*}
    \|\Lambda^{s}_{\omega}(uv)\|_{L^{2}_{\omega}}\lesssim
    \|u\|_{L^{\infty}_{\omega}}
    \|\Lambda^{s}_{\omega}v\|_{L^{2}_{\omega}}
    +
    \|\Lambda^{s}_{\omega}u\|_{L^{2}_{\omega}}
    \|v\|_{L^{\infty}_{\omega}},\qquad
    s>0
  \end{equation*}
  and an analogous one for $\Lambda^{s}_{\omega}|D|(uv)$.
  This would require a fair amount of calculus on the sphere
  $\mathbb{S}^{2}$, and here we preferred to use the
  conceptually
  much simpler algebra property of $H^{s}(\mathbb{S}^{n-1})$
  for $s>\frac{n-1}{2}$.
  \end{remark}

\end{document}